\documentclass[12pt]{amsart}
\usepackage{cases}
\usepackage{mathrsfs}
\usepackage{amsfonts}
\usepackage{amscd}

\usepackage{latexsym,amssymb}
\usepackage{a4wide}
\usepackage{epic,eepic,latexsym, amssymb, amscd, amsfonts, xypic, floatflt}
\usepackage{amsthm}
\usepackage{amsmath}
\usepackage{amscd}
\usepackage{graphicx}
\usepackage[mathscr]{eucal}
\input xy
\xyoption{all}
\usepackage{verbatim}

\numberwithin{equation}{section}
\begin{document}
\theoremstyle{plain}
\newtheorem{thm}{Theorem}[section]
\newtheorem{theorem}[thm]{Theorem}
\newtheorem{lemma}[thm]{Lemma}
\newtheorem{corollary}[thm]{Corollary}
\newtheorem{proposition}[thm]{Proposition}
\newtheorem{conjecture}[thm]{Conjecture}
\theoremstyle{definition}
\newtheorem{remark}[thm]{Remark}
\newtheorem{remarks}[thm]{Remarks}
\newtheorem{definition}[thm]{Definition}
\newtheorem{example}[thm]{Example}

\newcommand{\caA}{{\mathcal A}}
\newcommand{\caB}{{\mathcal B}}
\newcommand{\caC}{{\mathcal C}}
\newcommand{\caD}{{\mathcal D}}
\newcommand{\caE}{{\mathcal E}}
\newcommand{\caF}{{\mathcal F}}
\newcommand{\caG}{{\mathcal G}}
\newcommand{\caH}{{\mathcal H}}
\newcommand{\caI}{{\mathcal I}}
\newcommand{\caJ}{{\mathcal J}}
\newcommand{\caK}{{\mathcal K}}
\newcommand{\caL}{{\mathcal L}}
\newcommand{\caM}{{\mathcal M}}
\newcommand{\caN}{{\mathcal N}}
\newcommand{\caO}{{\mathcal O}}
\newcommand{\caP}{{\mathcal P}}
\newcommand{\caQ}{{\mathcal Q}}
\newcommand{\caR}{{\mathcal R}}
\newcommand{\caS}{{\mathcal S}}
\newcommand{\caT}{{\mathcal T}}
\newcommand{\caU}{{\mathcal U}}
\newcommand{\caV}{{\mathcal V}}
\newcommand{\caW}{{\mathcal W}}
\newcommand{\caX}{{\mathcal X}}
\newcommand{\caY}{{\mathcal Y}}
\newcommand{\caZ}{{\mathcal Z}}
\newcommand{\fA}{{\mathfrak A}}
\newcommand{\fB}{{\mathfrak B}}
\newcommand{\fC}{{\mathfrak C}}
\newcommand{\fD}{{\mathfrak D}}
\newcommand{\fE}{{\mathfrak E}}
\newcommand{\fF}{{\mathfrak F}}
\newcommand{\fG}{{\mathfrak G}}
\newcommand{\fH}{{\mathfrak H}}
\newcommand{\fI}{{\mathfrak I}}
\newcommand{\fJ}{{\mathfrak J}}
\newcommand{\fK}{{\mathfrak K}}
\newcommand{\fL}{{\mathfrak L}}
\newcommand{\fM}{{\mathfrak M}}
\newcommand{\fN}{{\mathfrak N}}
\newcommand{\fO}{{\mathfrak O}}
\newcommand{\fP}{{\mathfrak P}}
\newcommand{\fQ}{{\mathfrak Q}}
\newcommand{\fR}{{\mathfrak R}}
\newcommand{\fS}{{\mathfrak S}}
\newcommand{\fT}{{\mathfrak T}}
\newcommand{\fU}{{\mathfrak U}}
\newcommand{\fV}{{\mathfrak V}}
\newcommand{\fW}{{\mathfrak W}}
\newcommand{\fX}{{\mathfrak X}}
\newcommand{\fY}{{\mathfrak Y}}
\newcommand{\fZ}{{\mathfrak Z}}

\newcommand{\bA}{{\mathbb A}}
\newcommand{\bB}{{\mathbb B}}
\newcommand{\bC}{{\mathbb C}}
\newcommand{\bD}{{\mathbb D}}
\newcommand{\bE}{{\mathbb E}}
\newcommand{\bF}{{\mathbb F}}
\newcommand{\bG}{{\mathbb G}}
\newcommand{\bH}{{\mathbb H}}
\newcommand{\bI}{{\mathbb I}}
\newcommand{\bJ}{{\mathbb J}}
\newcommand{\bK}{{\mathbb K}}
\newcommand{\bL}{{\mathbb L}}
\newcommand{\bM}{{\mathbb M}}
\newcommand{\bN}{{\mathbb N}}
\newcommand{\bO}{{\mathbb O}}
\newcommand{\bP}{{\mathbb P}}
\newcommand{\bQ}{{\mathbb Q}}
\newcommand{\bR}{{\mathbb R}}
\newcommand{\bT}{{\mathbb T}}
\newcommand{\bU}{{\mathbb U}}
\newcommand{\bV}{{\mathbb V}}
\newcommand{\bW}{{\mathbb W}}
\newcommand{\bX}{{\mathbb X}}
\newcommand{\bY}{{\mathbb Y}}
\newcommand{\bZ}{{\mathbb Z}}
\newcommand{\id}{{\rm id}}

\title[Strong integrality for full colored HOMFLYPT invariants]{A simple proof of the strong integrality for full colored HOMFLYPT
invariants}
\author[Shengmao Zhu]{Shengmao Zhu}
\address{Center of Mathematical Sciences \\Zhejiang University \\Hangzhou, 310027, China }
\email{zhushengmao@gmail.com}

\begin{abstract}
By using the HOMFLY skein theory. We prove a strong integrality
theorem for the reduced colored HOMFLYPT invariants defined by a
basis in the full HOMFLY skein of the annulus.
\end{abstract}

\maketitle



\section{Introduction}
Let $(\mathcal{L}_+,\mathcal{L}_-,\mathcal{L}_0)$ be the standard
notation for the Conway triple of link diagrams.  The reduced Jones
polynomial $J(\mathcal{L};q)$ of a link (diagram) $\mathcal{L}$
 can be determined by the following skein relation
\begin{align}
q^2J(\mathcal{L}_+;q)-q^{-2}J(\mathcal{L}_-;q)=(q-q^{-1})J(\mathcal{L}_0;q),
\end{align}
and the normalization condition $J(U;q)=1$, where $U$ denotes the
trivial knot, i.e. the unknot throughout this paper. By induction on
the number of the components of the link $\mathcal{L}$. It is easy
to obtain, for an orient link $\mathcal{L}$ with $L$ components,
$J(\mathcal{L};q)\in \mathbb{Z}[q^{\pm 2}]$ \ for  $L$  odd, and
$qJ(\mathcal{L};q)\in \mathbb{Z}[q^{\pm 2}]$\ if $L$ even.
 In particular, for
a knot $\mathcal{K}$ (i.e. a link with only one component), we have
$ J(\mathcal{K};q)\in \mathbb{Z}[q^{\pm 2}].$ The Jones polynomial
$J(\mathcal{L};q)$ can be viewed as the $U_{q}(sl_2)$ quantum group
invariant with the fundamental representations \cite{RT}. More
generally, for a simple complex Lie algebra $\mathfrak{g}$,  T. Le
\cite{Le} showed that,  with a suitable normalization,  the quantum
group $U_{q}(\mathfrak{g})$ invariants with any irreducible
representations lie in the ring $\mathbb{Z}[q^{\pm 2}]$ which is
called the strong integrality (see Theorem 2.2 in \cite{Le}).

In this paper, we formulate a similar strong integrality theorem for
the reduced colored HOMFLYPT invariant defined by a basis in the
full HOMFLY skein of the annulus. We use the results from the HOMFLY
skein theory due to  Morton, Aiston, Lukac etc \cite{Ai,AM,Morton,L}
instead of using the quantum group theory by \cite{Lu,An}. So our
proof is completely different from Le's \cite{Le}.

Let $\mathcal{H}(\mathcal{L};q,a)$ be the framed HOMFLYPT invariant
of $\mathcal{L}$ determined by the local relations as showed in
Figure 1. Let $\mathcal{C}$ be the full HOMFLY skein of the annulus,
we will recall the definition in Section 2. In \cite{HM}, Hadji and
H. R. Morton constructed the basis elements $Q_{\lambda,\mu}$ in the
full skein $\mathcal{C}$. For a link $\mathcal{L}$ with $L$
components, we choose $L$ basis elements
$Q_{\lambda^\alpha,\mu^\alpha}$, $\alpha=1,...,L$ in the skein
$\mathcal{C}$. We construct the satellite link
$\mathcal{L}\star\otimes_{\alpha=1}^L
Q_{\lambda^\alpha,\mu^\alpha}$. Then the full colored HOMFLYPT
invariant for a link $\mathcal{L}$ is given by
\begin{align}
&W_{[\lambda^1,\mu^1],[\lambda^2,\mu^2],..,[\lambda^L,\mu^L]}(\mathcal{L};q,a)\\\nonumber
&=q^{-\sum_{\alpha=1}^L(\kappa_{\lambda^\alpha}+\kappa_{\mu^\alpha})w(\mathcal{K}_\alpha)}
a^{-\sum_{\alpha=1}^L(|\lambda^\alpha|+|\mu^\alpha|)w(\mathcal{K}_{\alpha})}
\mathcal{H}(\mathcal{L}\star\otimes_{\alpha=1}^L
Q_{\lambda^\alpha,\mu^\alpha};q,a).
\end{align}
In particular, when all $\mu^\alpha=\emptyset$ for $\alpha=1,..,L$,
$W_{[\lambda^1,\emptyset],[\lambda^2,\emptyset],..,[\lambda^L,\emptyset]}(\mathcal{L};q,a)$
is the ordinary colored HOMFLYPT invariant
$W_{\lambda^1,\lambda^2,..,\lambda^L}(\mathcal{L};q,a)$ studied in
\cite{Zhu,CLPZ}. We refer to Section 2 and 3 for a review of the
HOMFLY skein theory and the definition of the full colored HOMFLYPT
invariant for an oriented link. By the definition,
$W_{[\lambda^1,\mu^1],[\lambda^2,\mu^2],..,[\lambda^L,\mu^L]}(\mathcal{L};q,a)$
is not a polynomial of $q^{\pm 1}, a^{\pm 1}$ in general, because of
the factors $(q^k-q^{-k})$ in the denominator. In order to kill
these factors, we introduce the notion of reduced full colored
HOMFLYPT invariants.

For the case of a knot $\mathcal{K}$, it is natural to define the
reduced full colored HOMFLYPT invariant of $\mathcal{K}$ by
\begin{align}
P_{[\lambda,\mu]}(\mathcal{K};q,a)=\frac{W_{[\lambda,\mu]}(\mathcal{K};q,a)}{W_{[\lambda,\mu]}(U;q,a)}.
\end{align}
Then, we have
\begin{theorem}
For any knot $\mathcal{K}$, $P_{[\lambda,\mu]}(\mathcal{K};q,a)\in
\mathbb{Z}[q^{\pm 2},a^{\pm 2}]$.
\end{theorem}
In fact, Theorem 1.1 is a special case of the following Theorem 1.2.
Theorem 1.1 can also be proved directly by using the integrality
theorem \cite{Morton}(i.e Theorem 1 in \cite{Morton}) and the
Theorem 1.3 in the following for the case of knot.

It seems natural to generalize the definition of the reduced full
colored HOMFLYPT invariants for a general link $\mathcal{L}$ with
$L$ components as follow:
\begin{align}
P_{[\lambda^1,\mu^1],..,[\lambda^L,\mu^L]}(\mathcal{L};q,a)=
\frac{W_{[\lambda^1,\mu^1],..,[\lambda^L,\mu^L]}(\mathcal{L};q,a)}{\prod_{\alpha=1}^LW_{[\lambda^\alpha,\mu^\alpha]}(U;q,a)}.
\end{align}
Unfortunately,
$P_{[\lambda^1,\mu^1],..,[\lambda^L,\mu^L]}(\mathcal{L};q,a)$ does
not lie in the ring $\mathbb{Z}[q^{\pm 1},a^{\pm 1}]$ in general.
The recent works  by S. Nawata et al. \cite{NRZ,GNSSS} motivate us
to define the following  $[\lambda^\alpha,\mu^\alpha]$-reduced
HOMFLYPT invariants for a link $\mathcal{L}$:
\begin{align}
Q_{[\lambda^1,\mu^1],...,[\lambda^L,\mu^L]}^{[\mathcal{\lambda}^\alpha,\mathcal{\mu}^\alpha]}(\mathcal{L};q,a)
=\frac{W_{[\lambda^1,\mu_1],...,[\lambda^L,\mu^L]}(\mathcal{L};q,a)}
{W_{[\lambda^\alpha,\mu^\alpha]}(U;q,a)} \ \ \text{for}\ \
\alpha=1,...,L.
\end{align}
However,
$Q_{[\lambda^1,\mu^1],...,[\lambda^L,\mu^L]}^{[\mathcal{\lambda}^\alpha,\mathcal{\mu}^\alpha]}(\mathcal{L};q,a)$
still contains the factors of the forms $(q^k-q^{-k})$ in
denominator. In order to kill these factors, we add a factor to
$Q_{[\lambda^1,\mu^1],...,[\lambda^L,\mu^L]}^{[\mathcal{\lambda}^\alpha,\mathcal{\mu}^\alpha]}(\mathcal{L};q,a)$,
and introduce the notation of the normalized
$[\lambda^\alpha,\mu^\alpha]$-reduced invariant as
\begin{align}
P_{[\lambda^1,\mu^1],...,[\lambda^L,\mu^L]}^{[\lambda^\alpha,\mu^\alpha]}(\mathcal{L};q,a)&=\prod_{\beta=1,\beta\neq
\alpha}^{L}\prod_{\substack{\rho^\beta,\nu^\beta\\|\lambda^\beta|-|\rho^\beta|=
|\mu^\beta|-|\nu^\beta|}}\prod_{x\in \rho^\beta,\nu^\beta}(a\cdot
q^{cn(x)})(q^{hl(x)}-q^{-hl(x)})\\\nonumber &\cdot
Q_{[\lambda^1,\mu^1],...,[\lambda^L,\mu^L]}^{[\mathcal{\lambda}^\alpha,\mathcal{\mu}^\alpha]}(\mathcal{L};q,a).
\end{align}
See Section 3.1 for the definitions of $cn(x), hl(x)$ in formula
(1.6).

The main goal of this paper is to prove the following:
\begin{theorem}
For any $\alpha=1,...,L$, we have
\begin{align}
P_{[\lambda^1,\mu^1],...,[\lambda^L,\mu^L]}^{[\lambda^\alpha,\mu^\alpha]}(\mathcal{L};q,a)\in
\mathbb{Z}[q^{\pm 2}, a^{\pm 2}].
\end{align}
\end{theorem}
We give some special cases of Theorem 1.2.  For example, when all
the partitions $\mu^\beta= \emptyset$, we let
\begin{align}
P_{\lambda^1,...,\lambda^L}^{\lambda^\alpha}(\mathcal{L};q,a):&=
P_{[\lambda^1,\emptyset],...,[\lambda^L,\emptyset]}^{[\lambda^\alpha,\emptyset]}(\mathcal{L};q,a)\\\nonumber
&=\prod_{\beta=1,\beta\neq \alpha}^{L}\prod_{x\in
\lambda^\beta}(a\cdot
q^{cn(x)})(q^{hl(x)}-q^{-hl(x)})\frac{W_{\lambda^1,...,\lambda^L}(\mathcal{L};q,a)}
{W_{\lambda^\alpha}(U;q,a)}.
\end{align}
Moreover, if we take $\lambda^1=\lambda^2=\cdots =\lambda
^L=(r^\rho)$, by some straight calculations, we have
\begin{align}
\prod_{x\in
(r^\rho)}(q^{hl(x)}-q^{-hl(x)})=\prod_{i=0}^{\rho-1}\frac{\{r+i\}!}{\{i\}!},
\ \prod_{x\in (r^\rho)}(a\cdot
q^{cn(x)})=a^{r\rho}q^{\frac{1}{2}\rho r(r-\rho)},
\end{align}
where $\{n\}!=\prod_{i=1}^n\{i\}$, and $\{i\}=(q^i-q^{-i})$.

Hence,
\begin{align}
P_{(r^\rho),...,(r^\rho)}^{(r^\rho)}(\mathcal{L};q,a)=\left(a^{r\rho}q^{\frac{1}{2}\rho
r(r-\rho)}\prod_{i=0}^{\rho-1}\frac{\{r+i\}!}{\{i\}!}\right)^{L-1}\frac{W_{(r^\rho),...,(r^\rho)}(\mathcal{L};q,a)}
{W_{(r^\rho)}(U;q,a)}.
\end{align}
Theorem 1.2 implies
$P_{(r^\rho),...,(r^\rho)}^{(r^\rho)}(\mathcal{L};q,a)\in
\mathbb{Z}[q^{\pm 2},a^{\pm 2}]$. This is just a statement in
\cite{GNSSS}(see the last paragraph in page 42 in \cite{GNSSS}. In
fact, the definition of $P_{(r^\rho)}^{\text{fin}}(\mathcal{L};a,q)$
given by formula (3.2) in \cite{GNSSS} is equal to
 $(-1)^{r\rho
 (L-1)}P_{(r^\rho),...,(r^\rho)}^{(r^\rho)}(\mathcal{L};q,a)$ ).
According to the work \cite{GNSSS},
$P_{(r^\rho),...,(r^\rho)}^{(r^\rho)}(\mathcal{L};q,a)\in
\mathbb{Z}[q^{\pm 2},a^{\pm 2}]$ suggests that
$P_{(r^\rho),...,(r^\rho)}^{(r^\rho)}(\mathcal{L};q,a)$ may have a
categorification whose homology is finite dimensional. Similarly,
Theorem 1.2 provides us a fact that the categorification of the
general invariant
$P_{[\lambda^1,\mu^1],...,[\lambda^L,\mu^L]}^{[\lambda^\alpha,\mu^\alpha]}(\mathcal{L};q,a)$
has a finite dimensional homology.

In order to prove Theorem 1.2, firstly, we give a simple proof of
the following symmetries for the full colored HOMFLYPT invariants.
\begin{theorem}
For any link $\mathcal{L}$ with $L$-components, we have
\begin{align}
W_{[\lambda^1,\mu^1],..,[\lambda^L,\mu^L]}(\mathcal{L};q^{-1},a)&=(-1)^{\sum_{\alpha=1}^L(|\lambda^\alpha|+|\mu^\alpha|)}
W_{[(\lambda^1)^t,(\mu^1)^t],..,[(\lambda^L)^t,(\mu^L)^t]}(\mathcal{L};q,a),\\
W_{[\lambda^1,\mu^1],..,[\lambda^L,\mu^L]}(\mathcal{L};-q,a)&=(-1)^{\sum_{\alpha=1}^L(|\lambda^\alpha|+|\mu^\alpha|)}
W_{[\lambda^1,\mu^1],..,[\lambda^L,\mu^L]}(\mathcal{L};q,a),
\\
W_{[\lambda^1,\mu^1],..,[\lambda^L,\mu^L]}(\mathcal{L};q,-a)&=(-1)^{\sum_{\alpha=1}^L(|\lambda^\alpha|+|\mu^\alpha|)}
W_{[\lambda^1,\mu^1],..,[\lambda^L,\mu^L]}(\mathcal{L};q,a).
\end{align}
\end{theorem}
\begin{remark}
When all the $\mu^\alpha=\emptyset$, it is well-known that the above
symmetries hold for the ordinary colored HOMFLYPT invariants
$W_{\lambda^1,..,\lambda^L}(\mathcal{L};q,a)$ of a link
$\mathcal{L}$:
\begin{align}
W_{\vec{\lambda}}(\mathcal{L};q^{-1},t)&=(-1)^{\|\vec{\lambda}\|}W_{\vec{\lambda}^{t}}(\mathcal{L};q,t),\\
W_{\vec{\lambda}}(\mathcal{L};-q^{-1},t)&=W_{\vec{\lambda}^t}(\mathcal{L};q,t),\\
W_{\vec{\lambda}}(\mathcal{L};q,-t)&=(-1)^{\|\vec{\lambda}\|}W_{\vec{\lambda}}(\mathcal{L};q,t).
\end{align}
The first proof of these symmetries was given in the paper
\cite{LP1,LP2}, a different proof for (1.14) and (1.15) was given in
\cite{Zhu}. Then, in the joint paper with Q. Chen, K. Liu and P.
Peng \cite{CLPZ}, we provide a new simple proof for
(1.14),(1.15),(1.16). Recent,  D. Tubbenhauer, P. Vaz and P. Wedrich
\cite{TVW} also give a different proof for (1.14). In this paper, we
follow the method used in \cite{CLPZ} to prove Theorem 1.3 in
Section 4.
\end{remark}
Then, by using the integrality theorem in \cite{Morton} (i.e.
Theorem 1.1 in \cite{Morton}) with a slight modification,  we finish
the proof of Theorem 1.2  in Section 5.

\bigskip

{\bf Acknowledgements.} The author appreciate the collaborations
with Qingtao Chen, Kefeng Liu and Pan Peng in this area and many
valuable discussions with them within the past years.

\section{The Skein models}
Given a planar surface $F$, the framed HOMFLY skein $\mathcal{S}(F)$
of $F$ is the $\Lambda$-linear combination of orientated tangles in
$F$, modulo the two local relations as showed in figure \ref{Skein},
where $z=q-q^{-1}$,
\begin{figure}[!htb]
\begin{center}
\includegraphics[width=180 pt]{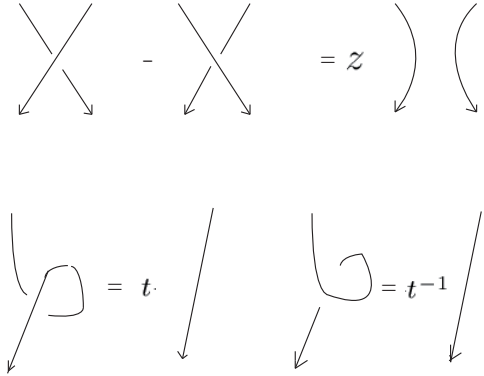}\caption{\label{Skein}Two local relations}
\end{center}
\end{figure}
the coefficient ring $\Lambda=\mathbb{Z}[q^{\pm 1}, a^{\pm 1} ]$
with the elements $q^{k}-q^{-k}$ admitted as denominators for $k\geq
1$. The local relation showed in figure \ref{Removal}
\begin{figure}[!htb]
\begin{center}
\includegraphics[width=100 pt]{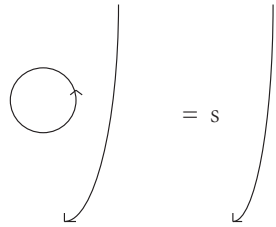}\caption{\label{Removal}Removal of a null-homotopic closed curve}
\end{center}
\end{figure}
is a consequence of the above relations. It follows that the removal
of a null-homotopic closed curve without crossings is equivalent to
time a scalar $s=\frac{a-a^{-1}}{q-q^{-1}}$.

\subsection{The plane} When $F=\mathbb{R}^2$, it is easy to follow
that every element in $\mathcal{S}(F)$ can be represented as a
scalar in $\Lambda$. For a link $\mathcal{L}$ with a diagram
$D_{\mathcal{L}}$, the resulting scalar $\langle D_{\mathcal{L}}
\rangle \in \Lambda$ is the framed HOMFLYPT polynomial
 of the link $\mathcal{L}$. In the
following, we will also use the notation
$\mathcal{H}(\mathcal{L};q,a)$ to denote the $\langle
D_{\mathcal{L}}\rangle$ for simplicity. In particular, as to the
unknot $U$, we have $\mathcal{H}(U;q,a)=\frac{a-a^{-1}}{q-q^{-1}}$.

\subsection{The annulus}
Let $\mathcal{C}$ be the HOMFLY skein of the annulus, i.e.
$\mathcal{C}=\mathcal{S}(S^1\otimes I)$. $\mathcal{C}$ is a
commutative algebra with the product induced by placing one annulus
outside another. $\mathcal{C}$ is freely generated by the set
$\{A_m: m\in \mathbb{Z}\}$, $A_m$ for $m\neq 0$ is the closure of
the braid $\sigma_{|m|-1}\cdots \sigma_2\sigma_1$, and $A_0$ is the
empty diagram \cite{Turaev2}. It follows that $\mathcal{C}$ contains
two subalgebras $\mathcal{C}_{+}$ and $\mathcal{C}_{-}$ which are
generated by $\{A_m: m\in \mathbb{Z}, m\geq 0\}$ and $\{A_m:m\in
\mathbb{Z}, m\leq 0\}$.  The algebra $\mathcal{C}_+$ is spanned by
the subspace $\mathcal{C}_{n,0}$. There is a good basis
$\{Q_{\lambda}\}$ of $\mathcal{C}_+$ consisting of the closures of
certain idempotents of Hecke algebra $H_{n,0}$ \cite{L}.

In \cite{HM}, R. Hadji and H. Morton constructed the basis elements
$\{Q_{\lambda,\mu}\}$ explicitly for $\mathcal{C}$. We will review
this construction in next section.

\section{Full colored HOMFLYPT invariants}
\subsection{Partitions and symmetric functions}
A partition $\lambda$ is a finite sequence of positive integers
$(\lambda_1,\lambda_2,..)$ such that
$
\lambda_1\geq \lambda_2\geq\cdots
$
The length of $\lambda$ is the total number of parts in $\lambda$
and denoted by $l(\lambda)$. The degree of $\lambda$ is defined by
$
|\lambda|=\sum_{i=1}^{l(\lambda)}\lambda_i.
$
If $|\lambda|=d$, we say $\lambda$ is a partition of $d$ and denoted
as $\lambda\vdash d$. The automorphism group of $\lambda$, denoted
by Aut($\lambda$), contains all the permutations that permute parts
of $\lambda$ by keeping it as a partition. Obviously, Aut($\lambda$)
has the order
$
|\text{Aut}(\lambda)|=\prod_{i=1}^{l(\lambda)}m_i(\lambda)!
$
where $m_i(\lambda)$ denotes the number of times that $i$ occurs in
$\lambda$. We can also write a partition $\lambda$ as
$
\lambda=(1^{m_1(\lambda)}2^{m_2(\lambda)}\cdots).
$
We denote by $\mathcal{P}$ the
set of all the partitions.

Every partition can be identified as a Young diagram. The Young
diagram of $\lambda$ is a graph with $\lambda_i$ boxes on the $i$-th
row for $j=1,2,..,l(\lambda)$, where we have enumerate the rows from
top to bottom and the columns from left to right. The $j$-th box in
the $i$-th row has the coordinates $(i,j)$. The content $cn(x)$ of
the box $x=(i,j)$ is defined to be $j-i$.

Given a partition $\lambda$, we define the conjugate partition
$\lambda^t$ whose Young diagram is the transposed Young diagram of
$\lambda$ which is derived from the Young diagram of $\lambda$ by
reflection in the main diagonal. For the box $x=(i,j)\in \lambda$,
the hook length is defined to be
$hl(x)=\lambda_i+\lambda_j^{t}-i-j+1$.

The following numbers associated with a given partition $\lambda$
are used frequently in this paper: $
z_\lambda=\prod_{j=1}^{l(\lambda)}j^{m_{j}(\lambda)}m_j(\lambda)!, \
\kappa_{\lambda}=\sum_{j=1}^{l(\lambda)}\lambda_j(\lambda_j-2j+1). $
Obviously, $\kappa_\lambda$ is an even number and
$\kappa_\lambda=-\kappa_{\lambda^t}$.

The $m$-th complete symmetric function $h_m$ and elementary
symmetric function are defined by the following two generating
functions
\begin{align}
H(t)=\sum_{m\geq 0}h_mt^m=\prod_{i\geq 1}\frac{1}{(1-x_it)}, \
E(t)=\sum_{m\geq 0}e_mt^m=\prod_{i\geq 1}(1+x_it),
\end{align}
respectively.

The power sum symmetric function of infinite variables
$x=(x_1,..,x_N,..)$ is defined by $ p_{n}(x)=\sum_{i}x_i^n. $ Given
a partition $\lambda$, we define $
p_\lambda(x)=\prod_{j=1}^{l(\lambda)}p_{\lambda_j}(x). $ The Schur
function $s_{\lambda}(x)$ is determined by the Frobenius formula
\begin{align}
s_\lambda(x)=\sum_{|\mu|=|\lambda|}\frac{\chi_{\lambda}(C_\mu)}{z_\mu}p_\mu(x).
\end{align}
where $\chi_\lambda$ is the character of the irreducible
representation of the symmetric group $S_{|\mu|}$ corresponding to
$\lambda$. $C_\mu$ denotes the conjugate class of symmetric group
$S_{|\mu|}$ corresponding to partition $\mu$. The orthogonality of
character formula gives
\begin{align}
\sum_A\frac{\chi_A(C_\mu) \chi_A(C_\nu)}{z_\mu}=\delta_{\mu \nu}.
\end{align}

For $\lambda,\mu,\nu\in \mathcal{P}$, we define the
littlewood-Richardson coefficient $c_{\lambda,\mu}^{\nu}$ as
\begin{align}
s_{\lambda}(x)s_{\mu}(x)=\sum_{\nu}c_{\lambda,\mu}^{\nu}s_{\nu}(x).
\end{align}
It is easy to see that $c_{\lambda,\mu}^{\nu}$ can be expressed by
the characters of symmetric group by using the Frobenius formula
\begin{align}
c_{\lambda,\mu}^{\nu}=\sum_{\rho,\tau}\frac{\chi_{\lambda}(C_{\rho})}{z_{\rho}}\frac{\chi_{\lambda}(C_{\tau})}{z_{\tau}}\chi_{\nu}(C_{\rho\cup
\tau}).
\end{align}

\subsection{Basic elements in $\mathcal{C}$}
Given a permutation $\pi\in S_m$ with the length $l(\pi)$, let
$\omega_\pi$ be the positive permutation braid associated to $\pi$.
We have $l(\pi)=w(\omega_\pi)$, the writhe number of the braid
$\omega_{\pi}$.

 We define the quasi-idempotent element in Hecke alegbra $H_m$,
\begin{align}
a_m=\sum_{\pi\in S_m}q^{l(\pi)}\omega_{\pi}
\end{align}
Let element $h_m\in \mathcal{C}_{m,0}$ be the closure of the
elements $\frac{1}{\alpha_m}a_m\in H_m$, i.e
$h_m=\frac{1}{\alpha_m}\hat{a}_m$. Where $\alpha_m$ is determined by
the equation $a_ma_m=\alpha_m a_m$, it gives
$\alpha_m=q^{m(m-1)/2}\prod_{i=1}^m\frac{q^i-q^{-i}}{q-q^{-1}}$.

The skein $\mathcal{C}_+$ ($\mathcal{C}_-$) is spanned by the
monomials in $\{h_m\}_{m\geq 0}$ ($\{h_k^*\}_{k\geq 0}$). The whole
skein $\mathcal{C}$ is spanned by the monomials in $\{h_m,
h_k^*\}_{m,k\geq 0}$. $\mathcal{C}_+$ can be regarded as the ring of
symmetric functions in variables $x_1,..,x_N,..$ with the
coefficient ring $\Lambda$ \cite{L2}. In this situation,
$\mathcal{C}_{m,0}$ consists of the homogeneous functions of degree
$m$. The power sum $P_m=\sum x_i^m$ are symmetric functions which
can be represented in terms of the complete symmetric functions,
hence $P_m\in \mathcal{C}_{m,0}$. Moreover, the following identity
was first obtained in \cite{Ai}, and see \cite{Mor} for a different
proof:
\begin{align}
[m]P_m=X_m=\sum_{j=0}^{m-1}A_{m-1-j,j}.
\end{align}
where $[m]=\frac{q^m-q^{-m}}{q-q^{-1}}$ and $A_{i,j}$ is the closure
of the braid $\sigma_{i+j}\sigma_{i+j-1}\cdots
\sigma_{j+1}\sigma_{j}^{-1}\cdots \sigma_1^{-1}$. Given a partition
$\mu$, we define
\begin{align}
P_{\mu}=\prod_{i=1}^{l(\mu)}P_{\mu_i}.
\end{align}

Then, in \cite{HM}, Hadji and Morton constructed the basis
$Q_{\lambda,\mu}$ on the whole skein $\mathcal{C}$ as follow. Given
two partitions $\lambda, \mu$ with $l$ and $r$ parts. We first
construct a $(l+r)\times (l+r)$-matrix $M_{\lambda,\mu}$ with
entries in $\{h_m, h_k^*\}_{m,k\in \mathbb{Z}}$ as follows, where we
have let $h_{m}=0$, if $m<0$ and $h_{k}^*=0$ if $k <0$.

\begin{align}
M_{\lambda,\mu}=
\begin{pmatrix}
h_{\mu_{r}}^* & h_{\mu_{r}-1}^* & \cdots & h_{\mu_{r}-r-l+1}^* \\
h_{\mu_{r-1}+1}^* & h_{\mu_{r-1}}^* & \cdots & h_{\mu_{r-1}-r-l}^*\\
\cdot & \cdot & \cdots & \cdot \\
h_{\mu_1+(r-1)}^* & h_{\mu_1+(r-2)}^* & \cdots & h_{\mu_1-l}^*\\
h_{\lambda_1-r} & h_{\lambda_1-(r-1)} & \cdots & h_{\lambda_1+l-1}\\
\cdot & \cdot & \cdots & \cdot \\
h_{\lambda_l-l-r+1} & h_{\lambda_l-s-r+2} & \cdots & h_{\lambda_l}
\end{pmatrix}
\end{align}
It is easy to note that the subscripts of the diagonal entries in
the $h$-rows are the parts $\lambda_1,\lambda_2,...,\lambda_l$ of
$\lambda$ in order, while the subscripts of the diagonal entries in
the $h^*$-rows are the parts $\mu_1,\mu_2,..,\mu_r$ of $\mu$ in
reverse order.

Then, $Q_{\lambda,\mu}$ is defined as the determinant of the matrix
$M_{\lambda,\mu}$, i.e.
\begin{align}
Q_{\lambda,\mu}=\det M_{\lambda,\mu}.
\end{align}
Usually, we  write $Q_{\lambda}=Q_{\lambda,\emptyset}$ and
$Q_{\mu}^*=Q_{\emptyset,\mu}$, we have
\begin{align}
Q_{\lambda,\mu}=\sum_{\sigma,\rho,\nu}(-1)^{|\sigma|}c_{\sigma,\rho}^{\lambda}c_{\sigma^t,\nu}^{\mu}Q_{\rho}
Q_{\nu}^*.
\end{align}
The Frobenius formula (3.2) implies
\begin{align}
Q_{\lambda}=\sum_{\mu}\frac{\chi_{\lambda}(\mu)}{z_{\mu}}P_{\mu}.
\end{align}

\subsection{Full colored HOMFLYPT invariants}
Let $\mathcal{L}$ be a framed link with $L$ components with a fixed
numbering. For diagrams $Q_1,..,Q_L$ in the skein model of annulus
with the positive oriented core $\mathcal{C}$, we define the
decoration of $\mathcal{L}$ with $Q_1,..,Q_L$ as the link
\begin{align}
\mathcal{L}\star \otimes_{\alpha=1}^{L} Q_\alpha
\end{align}
which derived from $\mathcal{L}$ by replacing every annulus
$\mathcal{L}$ by the annulus with the diagram $Q_\alpha$ such that
the orientations of the cores match. Each $Q_\alpha$ has a small
backboard neighborhood in the annulus which makes the decorated link
$\mathcal{L}\otimes_{\alpha=1}^{L}Q_\alpha$ into a framed link.

In particular, when $Q_{\lambda^\alpha,\mu^\alpha}\in
\mathcal{C}_{d_\alpha,t_\alpha}$, where $\lambda^\alpha, \mu^\alpha$
are the partitions of positive integers $d_\alpha$ and $t_\alpha$
respectively, for $\alpha=1,..,L$.

\begin{definition}
The full colored HOMFLYPT invariant of $\mathcal{L}$ is defined as
follow:
\begin{align}
&W_{[\lambda^1,\mu^1],[\lambda^2,\mu^2],..,[\lambda^L,\mu^L]}(\mathcal{L};q,a)\\\nonumber
&=q^{-\sum_{\alpha=1}^L(\kappa_{\lambda^\alpha}+\kappa_{\mu^\alpha})w(\mathcal{K}_\alpha)}
t^{-\sum_{\alpha=1}^L(|\lambda^\alpha|+|\mu^\alpha|)w(\mathcal{K}_{\alpha})}
\mathcal{H}(\mathcal{L}\star\otimes_{\alpha=1}^L
Q_{\lambda^\alpha,\mu^\alpha};q,a).
\end{align}
$W_{[\lambda^1,\mu^1],[\lambda^2,\mu^2],..,[\lambda^L,\mu^L]}(\mathcal{L};q,a)$
is a framing independent invariant of link $\mathcal{L}$. In fact,
by the result in \cite{GT}, the framing factor for $Q_{\lambda,\mu}$
is $q^{\kappa_{\lambda}+\kappa_{\mu}} t^{|\lambda|+|\mu|}$.
\end{definition}

\section{The symmetries}
Before giving the proof of Theorem 1.3, we need an observation which
is an easy consequence of the formula (3.7). We introduce the
notation $\{m\}=q^{m}-q^{-m}$, and for a partition $\lambda\in
\mathcal{P}$, we let
$\{\lambda\}=\prod_{i=1}^{l(\lambda)}\{\lambda_i\}$.
\begin{lemma}
Given any two partitions $\lambda$ and $\mu$, for a knot
$\mathcal{K}$, we have
\begin{align}
\{\lambda\}\{\mu\}\mathcal{H}(\mathcal{K}\star
P_{\lambda}P_{\mu}^*;q,a)\in \mathbb{Z}[(q-q^{-1})^2,a^{\pm 1}].
\end{align}
\end{lemma}
\begin{proof}
By using the skein relations in Figure 1 for framed HOMFLYPT
invariant $\mathcal{H}(\mathcal{L};q,a)$, we have, for a link
$\mathcal{L}$ with $L$ components,
\begin{align}
(q-q^{-1})^L\mathcal{H}(\mathcal{L};q,a) \in
\mathbb{Z}[(q-q^{-1})^2,a^{\pm 1}].
\end{align}
The formula (3.7) gives:
\begin{align}
\{m\}P_m=(q-q^{-1})X_m=(q-q^{-1})\sum_{j=0}^{m-1}A_{m-1-j,j}.
\end{align}
Therefore, by using formula (4.2), we get
\begin{align}
\{m\}\mathcal{H}(\mathcal{K}\star
P_m;q,a)=(q-q^{-1})\sum_{j=0}^{m-1}\mathcal{H}(\mathcal{K}\star
A_{m-1-j,j};q,a)\in \mathbb{Z}[(q-q^{-1})^2,a^{\pm 1}],
\end{align}
since $\mathcal{K}\star A_{m-1-j,j}$ is an one-component link for
$j=0,..,m-1$. Similarly, it is straightforward to obtain the formula
(4.1) in the same way.
\end{proof}

By using the above integrality result, we prove the following
symmetries for the full colored HOMFLYPT invariants.
\begin{theorem}
For any link $\mathcal{L}$ with $L$-components,
\begin{align}
W_{[\lambda^1,\mu^1],..,[\lambda^L,\mu^L]}(\mathcal{L};q^{-1},a)&=(-1)^{\sum_{\alpha=1}^L(|\lambda^\alpha|+|\mu^\alpha|)}
W_{[(\lambda^1)^t,(\mu^1)^t],..,[(\lambda^L)^t,(\mu^L)^t]}(\mathcal{L};q,a),\\
W_{[\lambda^1,\mu^1],..,[\lambda^L,\mu^L]}(\mathcal{L};-q,a)&=(-1)^{\sum_{\alpha=1}^L(|\lambda^\alpha|+|\mu^\alpha|)}
W_{[\lambda^1,\mu^1],..,[\lambda^L,\mu^L]}(\mathcal{L};q,a),
\\
W_{[\lambda^1,\mu^1],..,[\lambda^L,\mu^L]}(\mathcal{L};q,-a)&=(-1)^{\sum_{\alpha=1}^L(|\lambda^\alpha|+|\mu^\alpha|)}
W_{[\lambda^1,\mu^1],..,[\lambda^L,\mu^L]}(\mathcal{L};q,a).
\end{align}
\end{theorem}
\begin{proof}
For convenience, we only provide the proof for the case of a knot
$\mathcal{K}$.  The method can be easily generalized to the case for
a general link. By Lemma 4.1, for any two partitions $\tau$ and
$\delta$, we have,
\begin{align}
\{\tau\}\{\delta\}\mathcal{H}(\mathcal{K}\star
P_{\tau}P_{\delta}^*;q,a)\in \mathbb{Z}[(q-q^{-1})^2,a^{\pm 1}].
\end{align}
Combing with the identity $(\{\tau\}\{\delta\})_{q\rightarrow
-q}=(-1)^{|\tau|+|\delta|}\{\tau\}\{\delta\}$, we obtain
\begin{align}
\mathcal{H}(\mathcal{K}\star
P_{\tau}P_{\delta}^*;-q,a)=(-1)^{|\tau|+|\delta|}\mathcal{H}(\mathcal{K}\star
P_{\tau}P_{\delta}^*;q,a).
\end{align}
by the  formula (3.12),
\begin{align}
Q_{\rho}=\sum_{|\tau|=|\rho|}\frac{\chi_{\rho}(\tau)}{z_{\tau}}P_{\tau},
\
Q_{\nu}^*=\sum_{|\delta|=|\nu|}\frac{\chi_{\nu}(\delta)}{z_{\delta}}P_{\delta}^*.
\end{align}
So we have
\begin{align}
\mathcal{H}(\mathcal{K}\star
Q_{\rho}Q_{\nu}^*;-q,a)=(-1)^{|\rho|+|\nu|}\mathcal{H}(\mathcal{K}\star
Q_{\rho}Q_{\nu}^*;q,a).
\end{align}
By the formula (3.11) for $Q_{\lambda,\mu}$,
\begin{align}
\mathcal{H}(\mathcal{K}\star
Q_{\lambda,\mu};-q,a)&=\sum_{\sigma,\rho,\nu}(-1)^{|\sigma|}c_{\sigma,\rho}^{\lambda}c_{\sigma^t,\nu}^{\mu}\mathcal{H}(\mathcal{K}\star
Q_{\rho}Q_{\nu}^*;-q,a)\\\nonumber
&=\sum_{\sigma,\rho,\nu}(-1)^{|\sigma|}c_{\sigma,\rho}^{\lambda}c_{\sigma^t,\nu}^{\mu}(-1)^{|\rho|+|\nu|}\mathcal{H}(\mathcal{K}\star
Q_{\rho}Q_{\nu}^*;q,a)\\\nonumber
&=(-1)^{|\lambda|+|\mu|}\sum_{\sigma,\rho,\nu}(-1)^{|\sigma|}c_{\sigma,\rho}^{\lambda}c_{\sigma^t,\nu}^{\mu}\mathcal{H}(\mathcal{K}\star
Q_{\rho}Q_{\nu}^*;q,a)\\\nonumber
&=(-1)^{|\lambda|+|\mu|}\mathcal{H}(\mathcal{K}\star
Q_{\lambda,\mu};q,a),
\end{align}
where the third ``=" is from the observation that
$c_{\sigma,\rho}^{\lambda}=0$ if $|\lambda|\neq |\sigma|+|\rho|$.
Therefore, by the definition of the full colored HOMFLYPT invariants
$W_{[\lambda,\mu]}(\mathcal{L};q,a)$ in formula (3.14), we obtain
\begin{align}
W_{[\lambda,\mu]}(\mathcal{K},-q,a)=(-1)^{|\lambda|+|\mu|}W_{[\lambda,\mu]}(\mathcal{K},q,a),
\end{align}
since $\kappa_{\lambda}$ is even for any partition $\lambda$.

Similarly, by formula (4.8), we have
\begin{align}
\mathcal{H}(\mathcal{K}\star
P_{\tau}P_{\delta}^*;q^{-1},a)=(-1)^{l(\tau)+l(\delta)}\mathcal{H}(\mathcal{K}\star
P_{\tau}P_{\delta}^*;q,a).
\end{align}
By using the identities,
\begin{align}
\chi_{\rho^t}(\tau)=(-1)^{|\rho|-l(\tau)}\chi_{\rho}(\tau), \
\chi_{\nu^t}(\delta)=(-1)^{|\nu|-l(\delta)}\chi_{\rho}(\delta),
\end{align}
we have
\begin{align}
\mathcal{H}(\mathcal{K}\star
Q_{\rho}Q_{\nu}^*;q^{-1},a)=(-1)^{|\rho|+|\nu|}\mathcal{H}(\mathcal{K}\star
Q_{\rho^t}Q_{\nu^t}^*;q,a).
\end{align}
By the formula (3.5) for $c_{\sigma,\rho}^{\lambda}$, it is easy to
get
\begin{align}
c_{\sigma,\rho}^{\lambda}=c_{\sigma^t,\rho^t}^{\lambda^t}, \
c_{\sigma^t,\nu}^{\mu}=c_{\sigma,\nu^t}^{\mu^t}.
\end{align}
Therefore,
\begin{align}
\mathcal{H}(\mathcal{K}\star
Q_{\lambda,\mu};q^{-1},a)=(-1)^{|\lambda|+|\mu|}\mathcal{H}(\mathcal{K}\star
Q_{\lambda,\mu};q,a).
\end{align}
i.e.
\begin{align}
W_{[\lambda,\mu]}(\mathcal{K},q^{-1},a)=(-1)^{|\lambda|+|\mu|}W_{[\lambda^t,\mu^t]}(\mathcal{K},q,a),
\end{align}
since $\kappa_\lambda=-\kappa_{\lambda^t}$ and
$\kappa_{\mu}=-\kappa_{\mu^t}$.

Finally, in order to show the last identity (4.7),  by using  the
Lemma 4.3 proved in \cite{CLPZ}, we have (see the formula (4.22) in
\cite{CLPZ}):
\begin{align}
\mathcal{H}(\mathcal{K}\star
P_{\tau}P_{\delta}^*;q,-a)=(-1)^{(|\tau|+|\delta|)(w(\mathcal{K})-1)}\mathcal{H}(\mathcal{K}\star
P_{\tau}P_{\delta}^*;q,a).
\end{align}

Then it is direct to show
\begin{align}
\mathcal{H}(\mathcal{K}\star
Q_{[\lambda,\mu]};q,-a)=(-1)^{(|\lambda|+|\mu|)(w(\mathcal{K})-1)}\mathcal{H}(\mathcal{K}\star
Q_{[\lambda,\mu]};q,a),
\end{align}
i.e.
\begin{align}
W_{[\lambda,\mu]}(\mathcal{K},q,-a)=(-1)^{|\lambda|+|\mu|}W_{[\lambda,\mu]}(\mathcal{K},q,a).
\end{align}

\end{proof}

\section{Proof of the strong integrality}
Let us first consider the case of a knot $\mathcal{K}$. Recall the
definition of the reduced full colored HOMFLYPT invariant
\begin{align}
P_{[\lambda,\mu]}(\mathcal{K};q,a)=\frac{W_{[\lambda,\mu]}(\mathcal{K};q,a)}{W_{[\lambda,\mu]}(U;q,a)}.
\end{align}
The main result in \cite{Morton} (see Theorem 1 in \cite{Morton})
can be written in the following form:
\begin{proposition}[Morton]
For any knot $\mathcal{K}$, $P_{[\lambda,\mu]}(\mathcal{K};q,a)$ is
a 2-variable integer Laurent polynomial, i.e.
$P_{[\lambda,\mu]}(\mathcal{K};q,a)\in \mathbb{Z}[q^{\pm 1},a^{\pm
1}]$.
\end{proposition}

By the formulas (4.6) and (4.7) in Theorem 4.2, we immediately
obtain:
\begin{align}
P_{[\lambda,\mu]}(\mathcal{K};-q,a)=P_{[\lambda,\mu]}(\mathcal{K};q,a),\
P_{[\lambda,\mu]}(\mathcal{K};q,-a)=P_{[\lambda,\mu]}(\mathcal{K};q,a).
\end{align}
Therefore, $P_{[\lambda,\mu]}(\mathcal{K};q,a)$ satisfies the
following strong integrality:
\begin{theorem}
For any knot $\mathcal{K}$, $P_{[\lambda,\mu]}(\mathcal{K};q,a)\in
\mathbb{Z}[q^{\pm 2},a^{\pm 2}]$.
\end{theorem}
In fact, Proposition 5.1 can be generalized slightly to the case of
any link as follow. For a link $\mathcal{L}$ with $L$ components
$\mathcal{K}_1,...,\mathcal{K}_L$, if we only decorate the
$\alpha$-th component $\mathcal{K}_\alpha$ with
$Q_{\lambda^\alpha,\mu^\alpha}$, then its full colored HOMFLYPY
invariant is denoted by
$W_{[\lambda^\alpha,\mu^\alpha]}(\mathcal{L};q,a)$.
 Let
\begin{align}
P_{[\lambda^\alpha,\mu^\alpha]}(\mathcal{L};q,a)=\frac{W_{[\lambda^\alpha,\mu^\alpha]}(\mathcal{L};q,a)}{W_{[\lambda^\alpha,\mu^\alpha]}(U;q,a)},
\end{align}
we also have
\begin{proposition}
$P_{[\lambda^\alpha,\mu^\alpha]}(\mathcal{L};q,a)\in
\mathbb{Z}[q^{\pm 1},a^{\pm 1}]$.
\end{proposition}
\begin{proof}
A slight modification for the proof of Proposition 5.1 presented in
\cite{Morton}( see page 333-334) will give the proof of Proposition
5.3. In fact, given a link $\mathcal{L}$ with $L$ components
$\mathcal{K}_1,...,\mathcal{K}_L$, we denote it by
$\mathcal{L}=\mathcal{K}_1\bigvee \mathcal{K}_2\bigvee \cdots
\bigvee \mathcal{K}_L$. We cut the component $\mathcal{K}_\alpha$
open and get a 1-tangle, so we can draw $\mathcal{L}$ in the annulus
as the closure of this 1-tangle. Decorating $\mathcal{K}_\alpha$
with a diagram $Q_{\alpha}$ gives a diagram $\mathcal{K}_1\bigvee
\cdots\bigvee \mathcal{K}_\alpha \star Q_{\alpha}\bigvee \cdots
\bigvee \mathcal{K}_L$ in $\mathcal{C}$, it induces a linear map
$T_{\mathcal{K}_\alpha}^{\mathcal{L}}: \mathcal{C}\rightarrow
\mathcal{C}$. Refer to the page 332 in \cite{Morton} for this
construction in the case of $\mathcal{K}$. If $Q_\alpha$ is an
eigenvector of $T_{\mathcal{K}_\alpha}^{\mathcal{L}}$ with
eigenvalue $a_{\mathcal{K}_\alpha}^{\mathcal{L}}$. Then
$\mathcal{K}_1\bigvee \cdots\bigvee \mathcal{K}_\alpha \star
Q_{\alpha}\bigvee \cdots \bigvee
\mathcal{K}_L=T_{\mathcal{K}_\alpha}^{\mathcal{L}}(Q_\alpha)=a_{\mathcal{K}_\alpha}^{\mathcal{L}}Q_\alpha=a_{\mathcal{K}_\alpha}^{\mathcal{L}}U\star
Q_\alpha$ implies that
\begin{align}
a_{\mathcal{K}_\alpha}^{\mathcal{L}}=\frac{\mathcal{H}(\mathcal{K}_1\bigvee
\cdots\bigvee \mathcal{K}_\alpha \star Q_{\alpha}\bigvee \cdots
\bigvee \mathcal{K}_L;q,a)}{\mathcal{H}(U\star Q_\alpha;q,a)}.
\end{align}
The eigenvectors of $T_{\mathcal{K}_\alpha}^{\mathcal{L}}$ are given
by $Q_{\lambda,\mu}$ since the map
$T_{\mathcal{K}_\alpha}^{\mathcal{L}}$ commutes with the meridian
map $\phi$, see \cite{HM}. We denote the eigenvalue of
$T_{\mathcal{K}_\alpha}^{\mathcal{L}}$ corresponding to eigenvector
by $a_{\mathcal{K}_\alpha}^{\mathcal{L}}(\lambda,\mu)$, the proof of
Theorem 1 in \cite{Morton}(see page 334) shows that
\begin{align}
a_{\mathcal{K}_\alpha}^{\mathcal{L}}(\lambda,\mu)\in
\mathbb{Z}[q^{\pm 1},a^{\pm 1}].
\end{align}
\end{proof}
Moreover, Proposition 5.3 immediately implies that
\begin{proposition}
For any link $\mathcal{L}$, the full colored HOMFLYPT invariants
$W_{[\lambda^1,\mu^1],...,[\lambda^L,\mu^L]}(\mathcal{L};q,a)$
contains a factor $W_{[\lambda^\alpha,\mu^\alpha]}(U;q,a)$ where
$\alpha$ can be chosen to be $1,...,L$.
\end{proposition}
This phenomenon has been showed in \cite{NRZ} after doing many
detailed calculations for colored HOMFLYPT invariants with symmetric
representations by Chern-Simons theory.

Now, we recall the definition of normalized
$[\lambda^\alpha,\mu^\alpha]$-reduced full colored HOMFLYPT
invariant in Section 1.
\begin{align}
P_{[\lambda^1,\mu^1],...,[\lambda^L,\mu^L]}^{[\lambda^\alpha,\mu^\alpha]}(\mathcal{L};q,a)&=\prod_{\beta=1,\beta\neq
\alpha}^{L}\prod_{\substack{\rho^\beta,\nu^\beta\\|\lambda^\beta|-|\rho^\beta|=
|\mu^\beta|-|\nu^\beta|}}\prod_{x\in \rho^\beta,\nu^\beta}(a\cdot
q^{cn(x)})(q^{hl(x)}-q^{-hl(x)})\\\nonumber &\cdot
Q_{[\lambda^1,\mu^1],...,[\lambda^L,\mu^L]}^{[\mathcal{\lambda}^\alpha,\mathcal{\mu}^\alpha]}(\mathcal{L};q,a).
\end{align}
where
\begin{align}
Q_{[\lambda^1,\mu^1],...,[\lambda^L,\mu^L]}^{[\mathcal{\lambda}^\alpha,\mathcal{\mu}^\alpha]}(\mathcal{L};q,a)
=\frac{W_{[\lambda^1,\mu_1],...,[\lambda^L,\mu^L]}(\mathcal{L};q,a)}
{W_{[\lambda^\alpha,\mu^\alpha]}(U;q,a)} \ \ \text{for}\ \
\alpha=1,...,L.
\end{align}
We have the following strong integrality:
\begin{theorem}
For $\alpha=1,...,L$, we have
\begin{align}
P_{[\lambda^1,\mu^1],...,[\lambda^L,\mu^L]}^{[\lambda^\alpha,\mu^\alpha]}(\mathcal{L};q,a)\in
\mathbb{Z}[q^{\pm 2}, a^{\pm 2}].
\end{align}
\end{theorem}

\begin{proof}

By formula (3.11) and (3.12),
\begin{align}
Q_{\lambda,\mu}&=\sum_{\sigma,\rho,\nu}(-1)^{|\sigma|}c_{\sigma,\rho}^{\lambda}c_{\sigma^t,\nu}^{\mu}Q_{\rho}
Q_{\nu}^*\\\nonumber
&=\sum_{\sigma,\rho,\nu,\tau,\delta}(-1)^{|\sigma|}c_{\sigma,\rho}^{\lambda}c_{\sigma^t,\nu}^{\mu}\frac{\chi_{\rho}(\nu)}{z_\tau}
\frac{\chi_{\nu}(\delta)}{z_\delta}\frac{1}{\{\tau\}\{\delta\}}X_{\tau}X_{\delta}^*.
\end{align}
Therefore,
\begin{align}
\mathcal{\mathcal{L}}\star \otimes_{\beta=1}^{L}
Q_{\lambda^\alpha,\mu^\alpha}& =\mathcal{K}_1\star
Q_{\lambda^1,\mu^1}\bigvee \cdots\bigvee \mathcal{K}_\alpha \star
Q_{\lambda^\alpha,\mu^\alpha}\bigvee \cdots \bigvee
\mathcal{K}_L\star Q_{\lambda^L,\mu^L}\\\nonumber
&=\sum_{\tau^\beta,\delta^\beta, \beta\neq
\alpha}C_{\tau^1,\delta^1,..,\hat{\tau^\alpha},\hat{\delta^\alpha},..,\tau^L,\delta^L}
\prod_{\beta=1,\beta\neq
\alpha}\frac{1}{\{\tau^\beta\}\{\delta^\beta\}}\\\nonumber
&\cdot\mathcal{K}_1\star X_{\tau^1}X^*_{\delta^1}\bigvee
\cdots\bigvee \mathcal{K}_\alpha \star
Q_{[\lambda^\alpha,\mu^\alpha]}\bigvee \cdots \bigvee
\mathcal{K}_L\star X_{\tau^L}X^{*}_{\delta^L},
\end{align}
where
\begin{align}
C_{\tau^1,\delta^1,..,\hat{\tau^\alpha},\hat{\delta^\alpha},..,\tau^L,\delta^L}=
\sum_{\sigma^\beta,\rho^\beta,\nu^\beta,\tau^\tau,\delta^\tau,\beta\neq
\alpha}\prod_{\beta,\beta\neq
\alpha}(-1)^{|\sigma^\beta|}c_{\sigma^\beta,\rho^\beta}^{\lambda^\beta}c_{(\sigma^\beta)^t,\nu^\beta}^{\mu^\beta}
\frac{\chi_{\rho^\beta}(\nu^\beta)}{z_{\tau^\beta}}
\frac{\chi_{\nu^\beta}(\delta^\beta)}{z_{\delta^\beta}},
\end{align}
and $\hat{\tau^\alpha}, \hat{\delta^\alpha}$ denote the indexes
$\tau^\alpha, \delta^\alpha$ do not appear in the summation.

We denote the link
$\mathcal{L}_{\tau^1,\delta^1,..,\hat{\tau^\alpha},\hat{\delta^\alpha},..,\tau^L,\delta^L}=\mathcal{K}_1\star
X_{\tau^1}X^*_{\delta^1}\bigvee \cdots\bigvee \mathcal{K}_\alpha
\bigvee \cdots \bigvee \mathcal{K}_L\star
X_{\tau^L}X^{*}_{\delta^L}$. By formula (5.5),
\begin{align}
a_{\mathcal{K}_\alpha}^{
\mathcal{L}_{\tau^1,\delta^1,..,\hat{\tau^\alpha},\hat{\delta^\alpha},..,\tau^L,\delta^L}}(\lambda,\mu)\in
\mathbb{Z}[q^{\pm 1},a^{\pm 1}],
\end{align}
and
\begin{align}
\mathcal{H}(\mathcal{\mathcal{L}}\star \otimes_{\beta=1}^{L}
Q_{[\lambda^\alpha,\mu^\alpha]};q,a)&=\sum_{\tau^\beta,\delta^\beta,
\beta\neq
\alpha}C_{\tau^1,\delta^1,..,\hat{\tau^\alpha},\hat{\delta^\alpha},..,\tau^L,\delta^L}
\prod_{\beta=1,\beta\neq
\alpha}\frac{1}{\{\tau^\beta\}\{\delta^\beta\}}\\\nonumber &\cdot
a_{\mathcal{K}_\alpha}^{
\mathcal{L}_{\tau^1,\delta^1,..,\hat{\tau^\alpha},\hat{\delta^\alpha},..,\tau^L,\delta^L}}(\lambda,\mu)
\mathcal{H}(U\star Q_{[\lambda^\alpha,\mu^\alpha]};q,a).
\end{align}

Recall that we have two expressions for the colored HOMFLYPT
invariant for unknot $U$, the first one is given by (see
\cite{LP1}),
\begin{align}
\mathcal{H}(U\star
Q_{\lambda};q,a)=\sum_{\mu}\frac{\chi_{\lambda}(\mu)}{z_\mu}\prod_{i=1}^{l(\mu)}\frac{a^{\mu_i}-a^{-\mu_i}}{q^{\mu_i}-q^{-\mu_i}}.
\end{align}
Another one is given by (see Lemma 3.6.1 in page 51 in \cite{L2} ),
\begin{align}
\mathcal{H}(U\star Q_{\lambda};q,a)=\prod_{x\in
\lambda}\frac{a^{-1}q^{cn(x)}-aq^{-cn(x)}}{q^{h(x)}-q^{-h(x)}}.
\end{align}
Therefore, we have
\begin{align}
\prod_{x\in
\lambda}(q^{h(x)}-q^{-h(x)})\sum_{\mu}\frac{\chi_{\lambda}(\mu)}{z_\mu}\prod_{i=1}^{l(\mu)}\frac{1}{q^{\mu_i}-q^{-\mu_i}}\in
\mathbb{Z}[q^{\pm 1},a^{\pm 1}].
\end{align}
So for any $\beta\neq \alpha$,
\begin{align}
\prod_{x\in \rho^\beta,
\nu^\beta}\{hl(x)\}\sum_{\rho^\beta,\nu^\beta}\frac{\chi_{\rho^\beta}(\tau^\beta)}{z_{\tau^\beta}}
\frac{\chi_{\nu^\beta}(\delta^\beta)}{z_{\delta^\beta}}\frac{1}{\{\tau^\beta\}\{\delta^\beta\}}\in
\mathbb{Z}[q^{\pm 1},a^{\pm 1}].
\end{align}
Therefore, we have
\begin{align}
\prod_{\beta=1,\beta\neq
\alpha}^{L}\prod_{\substack{\rho^\beta,\nu^\beta\\|\lambda^\beta|-|\rho^\beta|=
|\mu^\beta|-|\nu^\beta|}}\prod_{x\in
\rho^\beta,\nu^\beta}\{hl(x)\}\sum_{\tau^\beta,\delta^\beta,
\beta\neq
\alpha}C_{\tau^1,\delta^1,..,\hat{\tau^\alpha},\hat{\delta^\alpha},..,\tau^L,\delta^L}
\prod_{\beta=1,\beta\neq
\alpha}\frac{1}{\{\tau^\beta\}\{\delta^\beta\}}\in \mathbb{Z}[q^{\pm
1},a^{\pm 1}].
\end{align}
\begin{align}
\prod_{\beta=1,\beta\neq
\alpha}^{L}\prod_{\substack{\rho^\beta,\nu^\beta\\|\lambda^\beta|-|\rho^\beta|=
|\mu^\beta|-|\nu^\beta|}}\prod_{x\in
\rho^\beta,\nu^\beta}\{hl(x)\}Q_{[\lambda^1,\mu^1],...,[\lambda^L,\mu^L]}^{[\mathcal{\lambda}^\alpha,\mathcal{\mu}^\alpha]}(\mathcal{L};q,a)\in
\mathbb{Z}[q^{\pm 1},a^{\pm 1}].
\end{align}
Hence,
\begin{align}
P_{[\lambda^1,\mu^1],...,[\lambda^L,\mu^L]}^{[\lambda^\alpha,\mu^\alpha]}(\mathcal{L};
q,a)\in \mathbb{Z}[q^{\pm 1},a^{\pm 1}].
\end{align}

On the other hand, by formulas (4.6) and (4.7), we have
\begin{align}
Q_{[\lambda^1,\mu^1],...,[\lambda^L,\mu^L]}^{[\mathcal{\lambda}^\alpha,\mathcal{\mu}^\alpha]}(\mathcal{L};\pm
q,\mp a)=(-1)^{\sum_{\beta\neq
\alpha}(|\lambda^\beta|+|\mu^\beta|)}Q_{[\lambda^1,\mu^1],...,[\lambda^L,\mu^L]}^{[\mathcal{\lambda}^\alpha,\mathcal{\mu}^\alpha]}(\mathcal{L};
q, a).
\end{align}
Combing the two expressions (5.14) and (5.15) for colored HOMFLYPT
invariant of the unknot $U$, for any partition $\lambda$, we obtain
the identity:
\begin{align}
(-1)^{|\lambda|}=(-1)^{\sum_{x\in \lambda}(hl(x)+cn(x))}.
\end{align}
Thus
\begin{align}
P_{[\lambda^1,\mu^1],...,[\lambda^L,\mu^L]}^{[\lambda^\alpha,\mu^\alpha]}(\mathcal{L};\pm
q,\mp
a)=P_{[\lambda^1,\mu^1],...,[\lambda^L,\mu^L]}^{[\lambda^\alpha,\mu^\alpha]}(\mathcal{L};q,a).
\end{align}
Combing formulas (5.20) and (5.23), we obtain
\begin{align}
P_{[\lambda^1,\mu^1],...,[\lambda^L,\mu^L]}^{[\lambda^\alpha,\mu^\alpha]}(\mathcal{L};
q, a)\in \mathbb{Z}[q^{\pm 2},a^{\pm 2}].
\end{align}
\end{proof}

\end{document}